\documentclass{amsart}
\usepackage{latexsym}
\usepackage{amsmath}
\usepackage{amssymb}
\usepackage[all]{xy}
\usepackage{easywncy}

\newtheorem{thm}{Theorem}[section]
\newtheorem{prop}[thm]{Proposition}

\newtheorem{cor}[thm]{Corollary}
\newtheorem{lem}[thm]{Lemma}
\theoremstyle{definition}
\newtheorem{defn}[thm]{Definition}

\newtheorem{assertion}[thm]{Assertion}

\theoremstyle{remark}
\newtheorem{rem}[thm]{Remark}
\newtheorem{ex}[thm]{Example}

\newcommand{\Z}{{\mathbb Z}}
\newcommand{\C}{{\mathcal C}}
\newcommand{\D}{{\mathcal D}}
\newcommand{\T}{{\mathcal T}}
\newcommand{\calK}{{\mathcal K}}
\newcommand{\E}{\mathcal{E}}
\newcommand{\F}{\mathcal{F}}
\newcommand{\mapright}[1]{%
\smash{\mathop{%
 \hbox to 1cm{\rightarrowfill}}\limits_{#1}}}
\newcommand{\maprightd}[2]{%
\smash{\mathop{%
 \hbox to 1.2cm{\rightarrowfill}}\limits^{#1}\limits_{#2}}}
\newcommand{\mapleft}[1]{%
\smash{\mathop{%
 \hbox to 1cm{\leftarrowfill}}\limits_{#1}}}
\newcommand{\mapleftu}[1]{%
\smash{\mathop{%
 \hbox to 0.8cm{\leftarrowfill}}\limits^{#1}}}
\newcommand{\maprightu}[1]{%
\smash{\mathop{%
 \hbox to 1cm{\rightarrowfill}}\limits^{#1}}}
\newcommand{\maprightud}[2]{%
\smash{\mathop{%
 \hbox to 1cm{\rightarrowfill}}\limits^{#1}_{#2}}}
\newcommand{\mapleftud}[2]{%
\smash{\mathop{%
 \hbox to 1cm{\leftarrowfill}}\limits^{#1}_{#2}}}

\newcommand{\eel}{\mathcyr{l}}
\newcommand{\QQQ}{\mathcal{Q}}

\newcommand{\id}{\mathop{\textit{id}}}

\newcommand{\Sets}{\mathsf{Sets}}
\newcommand{\xrightarrowrightarrow}[3][0.65ex]{\mathrel{\raisebox{#1}{\oalign{$\xrightarrow[]{#3}$\crcr$\xrightarrow[#2]{}$}}}}
\newcommand{\sss}{\mathfrak{s}}
\newcommand{\ttt}{\mathfrak{t}}

\newcounter{eqn}[section]

\def\theeqn{\textnormal{(\thesection.\arabic{eqn})}}

\def\eqnlabel#1{%
 \refstepcounter{eqn}%
 \label{#1}%
 \leqno{\theeqn}}

\begin{document}

\title{
A topos associated with a colored category \\
}

\author[K. Kuribayashi]{Katsuhiko KURIBAYASHI}
\address[Kuribayashi]{Department of Mathematical Sciences, Faculty of Science, Shinshu University, Matsumoto, Nagano 390-8621, Japan}
\email{kuri@math.shinshu-u.ac.jp}

\author[Y. Numata]{Yasuhiude NUMATA}
\address[Numata]{Department of Mathematical Sciences, Faculty of Science, Shinshu University, Matsumoto, Nagano 390-8621, Japan}
\email{nu@math.shinshu-u.ac.jp}


\footnote[0]{{\it 2010 Mathematics Subject Classification}: 18D99, 18F20, 16D90, 55N30, 05E30  \\
{\it Key words and phrases.} 
Schemoid, topos, cohomology, spectral sequence. }

\date{}

\begin{abstract} 
We show that a functor category whose domain is a colored category is a topos.
The topos structure enables us to introduce cohomology of colored categories including quasi-schemoids. 
If the given colored category arises from an association scheme, then the cohomology 
coincides with the group cohomology of the factor scheme by the thin residue.  Moreover, it is shown that 
the cohomology of a colored category relates to the standard representation of an association scheme via the Leray spectral sequence. 
\end{abstract}

\maketitle

\section{Introduction} 

Quasi-schemoids have been introduced in \cite{K-Matsuo} generalizing the notion of an {\it association scheme}  
\cite{B-I, P-Z, Z_book} from 
a small categorical point view. 
In a nutshell, the new object is a small category whose morphisms are colored with 
appropriate combinatorial data. Strong homotopy and representation theory for quasi-schemoids are developed 
in \cite{K} and \cite{K-Momose}, respectively. 

Once neglecting the combinatorial data in a quasi-schemoid, we have a category with colored morphisms. In what follows, such a category is called a {\it colored category}. 
The main theorem (Theorem \ref{thm:main}) in this article enables one to give a topos structure 
to a functor category whose domain is a colored category and whose objects are functors to the category of sets preserving colors; see Section 2 for the precise definition of 
the functor category. 
In consequence, appealing to the topos structure, 
we define cohomology of a colored category; see Definition \ref{defn:cohomology}.  
We have the inclusion functor from the functor category mentioned above 
to the usual functor category of the underlying category of the colored one. 
Theorem \ref{thm:main} also asserts that the inclusion gives rise to a geometric morphism of topoi 
whose direct image functor seems to be the {\it sheafification}. 

Applying the cohomology functor to an association scheme, we obtain the group cohomology of the factor scheme by the thin residue; 
see Proposition \ref{prop:H}. Then, one might think that the cohomology is not novel for association schemes.
However, our attempt to introduce cohomology of colored categories is thought of as the first step to study various cohomologies 
for such objects encompassing quasi-schemoids and hence association schemes; 
see Remark \ref{rem:assertions} (ii).  

A morphism between colored categories gives rise to 
a geometric morphism between the topoi associated with the colored categories.  
Thus the Leray spectral sequence in topos theory may allow us to investigate cohomology 
of a colored category. In particular, we apply the spectral sequence for considering cohomology 
of a colored poset; see Example \ref{ex:computaiton} below. 
Moreover, adjoint functors induced by a morphism from an association scheme $(X, S_X)$ to a colored category $(\C, S)$
connect the functor category of $(\C, S)$ with the module category over the Bose--Mesner algebra of $(X, S_X)$. 
Thus, for example, the cohomology of a colored category relates to the standard representation of 
an association scheme via the Leray spectral sequence; see Theorem \ref{thm:LSS}. 

The article is organized as follows. In Section 2, we describe our main theorem. Then 
{\it orthodox} cohomology of a colored category is defined. In Section 3, we prove the main theorem.  
In Section 4, geometric morphisms are investigated in our framework.  
An application and computational examples of cohomology of colored categories are also described. 
Section 5 considers the relationship mentioned above between cohomology of a colored category and 
the standard representation of an association scheme. In the end of the section, observations and expectations for our work are described. 

\section{Main results}
We begin by recalling the definition of a quasi-schemoid. 
A quasi-schemoid will be referred to as a schemoid in this article. 
Let $\C$ be a small category 
and  $S$ a partition of the set $mor(\C)$ of all morphisms in $\C$; that is, $mor(\C) =\coprod_{\sigma \in S}\sigma$. 
We call such a pair $(\C, S)$ a {\it colored category}. 
Moreover, a colored category $(\C, S)$ is called a {\it schemoid} 
if for a triple $\sigma, \tau, \mu \in S$ 
and for any morphisms $f$, $g$ in $\mu$, one has a bijection 
$$
(\pi_{\sigma\tau}^\mu)^{-1}(f) \cong (\pi_{\sigma\tau}^\mu)^{-1}(g), 
\eqnlabel{add-1}
$$ 
where $\pi_{\sigma\tau}^\mu : \pi_{\sigma\tau}^{-1}(\mu) \to \mu$ denotes 
the restriction of the composition map 
$$\pi_{\sigma\tau} : \sigma \times_{ob(\C)}\tau:=\{(f, g) \in \sigma \times \tau \mid s(f) = t(g)\} \to mor(\C).$$ 
The cardinality $p_{\sigma\tau}^\mu$  of the set $(\pi_{\sigma\tau}^\mu)^{-1}(f)$ is called a {\it structure constant}. 
We refer the reader to \cite[Section 2]{K-Matsuo}, \cite[Appendix A]{K-Momose} and \cite{Numata} for examples of schemoids. 
In case of a schemoid $(\C, S)$ with $\sharp mor(\C) < \infty$, we define the {\it Bose--Mesner algebra} associated with  $(\C, S)$ 
by using the structure constants; see \cite{K-Matsuo}.

Let $(\C, S)$ and $(\D, S')$ be colored categories. Then a functor $u : \C \to \D$ between underlying categories is called a 
{\it morphism of colored categories}, denoted $u :  (\C, S) \to (\D, S')$, if for $\sigma \in S$, there exists an element $\tau \in S'$ 
such that $u(\sigma) \subset \tau$. Observe that such an element $\tau$ is determined uniquely because $S'$ is a partition of 
$mor(\D)$.

Let $\T$ denote the category $\mathsf{Mod}$ of $\Z$-modules or the category $\mathsf{Sets}$ of sets. 
Though  $\T$ is not small, we regard it as a colored category whose morphisms have distinct colors. 
For morphisms $f$ and $g$ in a schemoid $(\C, S)$, 
we say that  $f$ is {\it equivalent} to $g$, denoted $f \sim_S g$, if $f$ and $g$ are contained in a common set 
$\sigma \in S$. 
For morphisms  $u, v : (\C,S) \to \T$ of colored categories, a natural transformation 
$\eta : u   \Rightarrow  v$ is called {\it locally constant} if $\eta_x = \eta_y$ whenever $id_x \sim_S id_y$.  
We define $\T^{(\C,S)}$ to be a category whose objects are morphisms of colored categories from $(\C,S)$ to $\T$ and 
whose morphisms are locally constant natural transformations; see \cite[Definition 2.3]{K-Momose} and the previous comments. 

We define an equivalence relation on the set of objects in $\C$. 
Let $\sim$ be a relation in $ob \ \C$ defined by $x \sim y$ if there exist a cell $\sigma \in S$ and morphisms $f$ and $g$ in $\sigma$ such that $(x, y) =(s(f), s(g))$ or $(x, y) =(t(f), t(g))$. 
We have the equivalence relation $\stackrel{0}{\sim}$ generated by the relation 
$\sim$ above. 
Let $\natural\T^{(\C, S)}$ be the wide subcategory of $\T^{(\C, S)}$ consisting of all morphisms $\eta$ satisfying 
the condition that $\eta_x = \eta_y$ if $x \stackrel{0}{\sim} y$. Observe that every morphism in $\natural\T^{(\C, S)}$ is locally constant 
and by definition, $ob (\natural\T^{(\C, S)}) = ob (\T^{(\C, S)})$. We have a sequence of inclusion functors 
$$
\natural\T^{(\C, S)} \subset \T^{(\C, S)} \subset \T^\C, 
$$
where $T^\C$ denotes the usual functor category on $\C$.  The first two categories coincide in some case. 

\begin{defn}\label{defn:mild_object}
A colored category $(\C, S)$ is called a {\it naturally colored category} if the following condition is satisfied: 
$id_{t(f)} \sim_S id_{t(g)}$ and  $id_{s(f)} \sim_S id_{s(g)}$ whenever 
$f$ and $g$ are in $\sigma$ for some $\sigma \in S$, .  A naturally colored category $(\C, S)$ 
is a {\it natural schemoid} if it moreover is a schemoid; that is, $(\C, S)$ is endowed with the bijection (2.1) for each triple of elements in $S$. 
\end{defn}

\begin{prop}\label{prop:equality} 
Let $(\C, S)$ be a naturally colored category. Then the functor category $\T^{(\C, S)}$ coincides with 
the subcategory $\natural\T^{(\C, S)}$. 
\end{prop}

\begin{proof}
Since $x\stackrel{0}{\sim}y$ if and only if $id_x\sim_S id_y$, it follows that each morphism in $\T^{(\C, S)}$ is in the wide subcategory. 
We have the result. 
\end{proof} 

In general, the category $\T^{(\C, S)}$ is larger than $\natural\T^{(\C, S)}$ while the classes of objects coincide; see Example \ref{ex:pullbacks}.  

We recall the definition of a {\it tame} schemoid introduced in \cite[page 229]{K-Momose}. 
For a schemoid $(\C, S)$, we consider the following conditions T(i), T(ii) and T(iii). \\

\noindent
T(i): The schemoid $(\C, S)$ is unital, namely, for  $J_0 :=  \{id_x\}_{x \in ob\C}$,
$$
\{id_x\}_{x \in ob\C} = \displaystyle{\bigcup_{\alpha \in S, \alpha\cap J_0\neq \varnothing}}\alpha. 
$$

\noindent
T(ii):  For any $\sigma \in S$ and $f, g \in \sigma$, there exist $\tau_1$ and $\tau_2$ in $S$ such that 
$id_{s(f)}, id_{s(g)} \in \tau_1$ and $id_{t(f)}, id_{t(g)} \in \tau_2$. \\

The third one is required to introduce a category $[\C]$ associated with a schmeoid $(\C, S)$,  
whose set of objects is defined by   
$$
ob [\C] = \{id_x\}_{x \in ob\C} \slash \sim_S \ =\{[x] \},  
$$ 
where we write $[x]$ for $[id_x]$.  Under the condition T(ii), for an element $\sigma \in S$, there exists a unique 
element $[x]$ in $ob[\C]$ such that $id_{s(f)} \in [x]$ for any $f \in \sigma$. In this case, we write $s(\sigma) \subset [x]$. 
Similarly, we write $t(\sigma) \subset [y]$ if $id_{t(f)} \in [y]$ for any $f \in \sigma$. 
Define a set of morphisms from $[x]$ to $[y]$ in the diagram $[\C]$ by 
$$
mor_{[\C]}([x], [y]) = \{\sigma \in S \mid s(\sigma) \subset [x], \ t(\sigma) \subset [y]\}.     
$$
\noindent
T(iii): For morphisms $[x] \stackrel{\sigma}{\longrightarrow} [y] \stackrel{\tau}{\longrightarrow} [z]$, 
there exist $f \in \sigma$ and 
$g \in \tau$ such that $s(g) = t(f)$. Moreover, there is a unique element $\mu = \mu(\tau, \sigma)$ in $S$ 
such that $p_{\tau \sigma}^{\mu} \geq 1$.  \\

It follows from \cite[Remark 3.1]{K-Momose} that 
the implication $\text{T(i)} \Rightarrow \text{T(ii)}$ holds. Observe that the condition T(ii) 
is nothing but that in the definition of a natural schemoid; see Definition \ref{defn:mild_object}.
A schemoid $(\C, S)$ is called {\it tame} if the conditions T(i) and T(iii) hold.  Thus a tame schemoid is natural.

\begin{lem} \label{lem:tame_qs}{\em (}\cite[Lemma 3.3]{K-Momose}{\em )} 
Let $(\C, S)$ be a tame schemoid. Then the diagram $[\C]$ is a category with the composite of morphisms 
defined by $\tau \circ \sigma = \mu(\tau, \sigma)$. 
\end{lem}

\begin{rem}\label{rem:mildness} The schemoids arising from a discrete group and an association scheme are natural; 
see the diagram (2.2). Moreover, the schemoids constructed by Boolean posets including 
an abstract simplicial complex are natural; see \cite{K-Momose, Numata}.  
A {\it coherent configuration} \cite{Higman} gives a natural schemoid via the same functor as $\jmath$ in the diagram (2.2). 
In fact, the implication mentioned above gives the result. 
\end{rem}

Tameness and naturality of a schemoid are described in terms of a {\it coloring map} used in \cite{Numata_2}; see Appendix B. 

Our main theorem (Theorem \ref{thm:main}) below asserts that the functor category $\natural\mathsf{Sets}^{(\C, S)}$ is a topos 
for every colored category $(\C, S)$. 
We recall the definition of a topos, which is described in the Giraud form; see, for example, \cite[\S 1]{M}, \cite[0.45 Theorem]{J} and \cite[page 577]{M-M}. 

\begin{defn}
A category $\E$ is said to be a  (Grothendieck) {\it topos} if it satisfies the Giraud axioms (G1), (G2), (G3) and (G4) below. 

\medskip
\noindent
(G1) The category $\E$ has finite limits. \\
(G2) All set-indexed sums exist in $\E$, and are compatible with every pullback construction. Moreover, sums are disjoint; that is, 
for a family $\{E_i\}_{i\in I}$ 
of objects in $\E$, the diagram 
$$
\xymatrix@C20pt@R20pt{ 
0 \ar[r] \ar[d]  &  E_i \ar[d] \\
E_j \ar[r] & \Sigma_{i\in I} E_i
}
$$
is a pullback for any $i$ and $j$, where $0$ denotes the initial object. 

A diagram (*): $\xymatrix@C20pt@R20pt{R \ar@<0.5ex>[r]^-{r} \ar@<-0.5ex>[r]_-{s} & E \ar[r]^f & F}$ in $\E$ is said to be {\it exact} if $f$ 
is the coequalizer of $r$ and $s$ and the diagram 
$$
\xymatrix@C20pt@R20pt{ 
R \ar[r]^s \ar[d]_r  &  E \ar[d]^f \\
E \ar[r]_f & F
}
$$
is a pullback. Moreover,  we say that the diagram (*) above is {\it stably exact} if the exactness is preserved under every pullback construction.  

A monomorphism $\xymatrix@C20pt@R20pt{R \ \ \ar@{>->}[r] &  E\times E}$ is said to be an {\it equivalence relation} 
if the induced inclusion 
$\text{Hom}_\E(T, R) \subset \text{Hom}_\E(T, E\times E) \cong \text{Hom}_\E(T, E) \times \text{Hom}_\E(T, E)$ 
is an equivalence relation on the set $\text{Hom}_\E(T, E)$ for every object $T$. 

\noindent
(G3)  
(i) For every epimorphism $E \to F$ in $\E$, the diagram 
$\xymatrix@C20pt@R20pt{E\times_F E \ar@<0.5ex>[r] \ar@<-0.5ex>[r] &E \ar[r] & F}$
is stably exact. \\
(ii) For every equivalence relation $\xymatrix@C20pt@R20pt{R \ \ \ar@{>->}[r] &  E\times E}$, there exists an object $E/R$ which fits into an exact diagram  
$\xymatrix@C20pt@R20pt{R \ar@<0.5ex>[r] \ar@<-0.5ex>[r] & E \ar[r] & E/R}$.

We call a set ${\mathcal I}$ of objects in $\E$ a {\it set of generators} if for distinct morphisms 
$f, g : X \to Y$, there exist an object $A$ in ${\mathcal I}$ and a morphism $h : A \to X$ such that $f\circ h \neq g\circ h$. \\
(G4) The category $\E$ has a set of generators.
\end{defn}

We also recall the definition of a morphism of topoi; see \cite[1.16 Definition]{J} and \cite[Chapter I, \S 1]{M}. 

\begin{defn}
A {\it geometric morphism} $f : \F \to \E$ of topoi consists of a pair of functors ({\it inverse} and {\it direct} image functors) 
$f^* :  \E \to \F$ and  $f_* :  \F \to \E$ with the following properties:
(i) $f^*$ is left adjoint to $f_*$; $f^* \dashv f_*$, (ii) $f^*$ commutes with finite limits. 
\end{defn}

We may write $(f^*, f_*)$ for such a geometric morphism $f$. 
Our main theorem is described as follows. 

\begin{thm}\label{thm:main}
Let $(\C, S)$ be a colored category. 
Then the functor category $\natural\mathsf{Sets}^{(\C, S)}$ is a topos. In consequence, 
the category of abelian group objects $Ab(\natural\mathsf{Sets}^{(\C, S)}) = \natural\mathsf{Mod}^{(\C, S)}$ in the topos 
$\natural\mathsf{Sets}^{(\C, S)}$ has enough injectives. Moreover, 
the inclusion functor $\iota : \natural\mathsf{Sets}^{(\C, S)} \to \mathsf{Sets}^{\C}$ gives rise to 
a geometric morphism of topoi
$$
f =(\iota, f_*) : \mathsf{Sets}^{\C} \to \natural\mathsf{Sets}^{(\C, S)}. 
$$
\end{thm}
The right adjoint $f_*$ in Theorem \ref{thm:main} behaves like the ``sheafification" 
since $\iota$ is the inclusion. In fact, for any ``presheaf" $F \in ob(\mathsf{Sets}^{\C})$, 
we have the functor $f_*(F) : (\C, S) \to \T$ which preserves colors.  

Theorem \ref{thm:main} enables us to define cohomology of colored categories according to the usual procedure. 

\begin{defn}\label{defn:cohomology}
Let $(\C, S)$ be a colored category. Cohomology $H^*((\C, S), M)$ of $(\C, S)$ with coefficients in $M$, 
which is an object in $\mathsf{Mod}^{(\C, S)}$ and hence in $\natural\mathsf{Mod}^{(\C, S)}$, 
is defined to be the right derived functor of 
$\text{Hom}_{\natural\mathsf{Mod}^{(\C, S)}}(\underline{\mathbb Z}, \ \ )$, namely,  
$$
H^*((\C, S), M) := \text{Ext}_{\natural\mathsf{Mod}^{(\C, S)}}^*(\underline{\mathbb Z}, M), 
$$
where $\underline{\mathbb Z}$ stands for the constant sheaf with values in ${\mathbb Z}$.
\end{defn}

\begin{rem}\label{rem:discreteSmod}
Let $\C$ be a small category and $\calK(\C)=(\C, S)$ the discrete schemoid; that is, 
the partition $S$ is given by $S=\{ \{f\} \}_{f \in mor(\C)}$; 
see \cite[Example 2.1]{K-Momose}. We see that $\mathsf{Sets}^{\calK(\C)}$ is the usual functor category 
$\mathsf{Sets}^{\C}$ and then $\mathsf{Sets}^{\calK(\C)}$ is a topos, which is the so-called {\it classifying topos} of $\C$. 
In particular, the Yoneda lemma enables us to verify that the axiom (G4) is satisfied in 
$\mathsf{Sets}^{\C}$; see, for example, \cite[page 11]{M}. Here we regard $\C$ as $(\C^{op})^{op}$. 
For a colored category $(\C, S)$, the representation functor $\text{Hom}_{\C}(x, \ )$ is {\it not} in  $\mathsf{Sets}^{(\C, S)}$ in general. Therefore, in order to prove Theorem \ref{thm:main}, 
we do not apply the same proof as that of the result above. 
\end{rem}

In order to prove Theorem \ref{thm:main}, we shall show that the wide subcategory $\natural\mathsf{Sets}^{(\C, S)}$ 
is isomorphic to a classifying topos. Such a topos is described below. 

We recall the equivalence relation  $\stackrel{0}{\sim}$ and denote by $I_0$ the quotient $ob \ \C/ \stackrel{0}{\sim}$.
 For a cell $\sigma \in S$, we define $s(\sigma) =[s(f)]$ and $t(\sigma)=[t(f)]$ if $f \in \sigma$.  
Observe that $s(\sigma)$ and $t(\sigma)$ are elements in $I_0$ and that these elements are determined 
independent of the choice of the morphism $f$ in $\sigma$. 

Let $M$ be the set of all finite sequences $\sigma_n\cdots \sigma_1$ with $\sigma_j \in S$ and 
$t(\sigma_i)=s(\sigma_{i+1}) \ \text{for}  \ 1 \leq i \leq n-1$. 
For elements $[x]$ and $[y]$ in the set $ob \ \C / \stackrel{0}{\sim}$, we define a subset $M_{[x][y]}$ of $M$ by 
$$
M_{[x][y]} := \{ \sigma_n\cdots \sigma_1 \in M \mid s(\sigma_1) =[x], t(\sigma_n)=[y] \}. 
$$
We define a relation $\stackrel{ob}{\sim}$ on $M_{[x][y]}$ by $u\sigma v \stackrel{ob}{\sim} u\mu v$ for $u, v \in M$ 
and $\sigma, \mu \in S$ if there exist objects $a$ and $b$ such that 
$id_a \in \sigma$ , $id_b \in \mu$ and $a \stackrel{0}{\sim} b$. 
Moreover, let $\stackrel{c}{\sim}$ be a relation defined by $u\mu\tau v \stackrel{c}{\sim} u\sigma v$ for $u, v \in M$ and 
$\mu, \tau, \sigma \in S$ provided there exist morphisms $l \in \mu$ and $k\in \tau$ such that $lk \in \sigma$. 
Then, we have an equivalence relation $\stackrel{1}{\sim}$ in $M_{[x][y]}$ generated by relations $\stackrel{ob}{\sim}$ and $\stackrel{c}{\sim}$.  
The concatenation of sequences gives rise to a well-defined composite 
$
M_{[y][z]}/\stackrel{1}{\sim} \times \ M_{[x][y]}/\stackrel{1}{\sim} \ \longrightarrow  M_{[x][z]}/\stackrel{1}{\sim}.
$
Thus, we have a small category $c[(\C, S)]$ whose set of objects is the quotient set $I_0$ and whose homset $\text{Hom}_{c[(\C, S)]}([x], [y])$ 
is the quotient $M_{[x][y]}/\stackrel{1}{\sim}$. Observe that $\sigma = id_{[x]}$ in $c[(\C, S)]$ whenever $id_x \in \sigma$. 

\begin{lem}\label{lem:pi}
Let $\pi : \C \to c[(\C, S)]$ be defined on objects by $\pi(x)=[x]$ and on morphisms by $\pi(f)=\sigma$, where $f \in \sigma$. Then $\pi$ 
is a functor.  
\end{lem}

The lemma is proved immediately. The following theorem is a key to proving our main theorem. 

\begin{thm} \label{thm:module} Let $(\C, S)$ be a colored category.  
Then the functor $\pi$ induces a functorial isomorphism $\pi^*: \mathsf{Sets}^{c[(\C, S)]} \to \natural\mathsf{Sets}^{(\C, S)}$ of categories.  
\end{thm}

Before proving Theorem \ref{thm:module}, we give comments on the category $c[(\C, S)]$. 

\begin{rem} \label{rem:monoids} 
Let $(\C, S)$ be a colored category. Suppose that $id_x\sim_S id_y$ for any objects $x$ and $y$ in $\C$. 
It is readily seen that $(\C, S)$ is a naturally colored category.  
Moreover,  we see that 
the category $c[(\C, S)]$ is the monoid generated by $S$ with relations such that $\sigma\tau =\mu$ if 
there exist composable morphisms $k \in \sigma$ and $h \in \tau$ with $kh \in \mu$.  
Observe that the relation $\sigma \stackrel{ob}{\sim} \tau$ implies the equality $\sigma = \tau$ in this case. 
Assume further that $(\C, S)$ is a schemoid. Then we have  
$$
c[(\C, S)] = \langle \sigma \in S \mid \sigma\tau =\mu  \ \text{if} \  p_{\sigma\tau }^\mu \neq 0 \rangle. 
$$
\end{rem}

Let $\mathsf{Gr}$, $\mathsf{AS}$, $\mathsf{Gpd}$, $\mathsf{Cat}$ and $q\mathsf{ASmd}$ be categories of groups, 
association schemes, groupoids, small categories and schemoids, respectively. The category $\mathsf{AS}$ has been introduced in \cite{H}. 
We here recall a commutative diagram of categories 
$$
\xymatrix@C35pt@R25pt{
\mathsf{Gpd} \ar[r]^-{{\widetilde S}( \ )} & 
q\mathsf{ASmd} \ar@<1ex>[r]^-{U}_-{\top} 
& \mathsf{Cat}, \ar@<1ex>[l]^-{\calK}  \\
\mathsf{Gr} \ar[u]^\imath \ar[r]^-{S( \ )}  & \mathsf{AS} \ar[u]_{\jmath} 
}
\eqnlabel{add-2}
$$ 
where $\imath : \mathsf{Gr} \to \mathsf{Gpd}$ is the natural fully faithful embedding  
and the functor $S( \ )$ assigns group-case association schemes to groups. 
Moreover, $\calK$ is a functor given by sending a small category to 
the discrete schemoid; see Remark \ref{rem:discreteSmod}. 
The functor $\jmath : \mathsf{AS} \to q\mathsf{ASmd}$ is indeed the composite of a fully faithful functor introduced 
in \cite[Example 2.6 (ii)]{K-Matsuo} and an inclusion functor. 
Observe that the functor $\calK$ is the left adjoint to the forgetful functor $U$; 
see \cite[Sections 2 and 3]{K-Matsuo} for more detail. 

The following result due to Hanaki \cite{Hanaki_privatecom} 
tells us what the category $c[\jmath (X, S)]$ for an association scheme $(X, S)$ is. The proof is postponed to Appendix A. 

\begin{prop}\label{prop:H}
Let $(X, S)$ be an association scheme and $(X, S)^{{\bf O}^\vartheta(S)}$ the factor scheme by the thin residue ${\bf O}^\vartheta(S)$; see \cite[2.3]{Z_book1}. 
Then there exists a functorial isomorphism 
$c[\jmath (X, S)]\cong (X, S)^{{\bf O}^\vartheta(S)}=: Quo(S)$ of groups. 
\end{prop}

\section{Proof of the main theorem}

Our proof of  Theorem \ref{thm:module} is essentially the same as those 
of the proofs of \cite[Theorems 2.1 and 2.2]{Numata_2}. 
For the reader, we describe it in our context.

\begin{lem}\label{lem:objects} For an object $F$ in $\mathsf{Sets}^{(\C, S)}$, 
if $x \stackrel{0}{\sim} y$, then $F(x) = F(y)$.
\end{lem}

\begin{proof}
Suppose that $(x, y) =(s(f), s(g))$ or $(x, y) =(t(f), t(g))$ for some $f$ and $g$ with 
$f\sim_S g$. Then $F(f) = F(g)$. This yields that $F(x) = F(y)$. 
\end{proof}

\begin{proof}[Proof of Theorem \ref{thm:module}] 
We first define a functor $\theta : \natural\mathsf{Sets}^{(\C, S)}\to \mathsf{Sets}^{c[(\C, S)]}$ by 
$$
(\theta F)([x]) =  F(x) \ \text{and} \  (\theta F)(\sigma_n\cdots \sigma_1) = F(f_n)\cdots F(f_1), 
$$
where $[x]$ denotes an object of $c[(\C, S)]$, 
$f_i \in \sigma_i$ and $F$ is an object in 
$\natural\mathsf{Sets}^{(\C, S)}$.  Lemma \ref{lem:objects} implies that $\theta F$ is well defined in the objects of $c[(\C, S)]$. 

We verify the well-definedness of $\theta F$ in the morphisms.  As for the composite, since $t(\sigma_i)=s(\sigma_{i+1})$, it follows that 
there exist maps $f_i \in \sigma_i$ and $f_{i+1} \in \sigma_{i+1}$ such that $[t(f_i)]=[s(f_{i+1})]$. 
In view of Lemma \ref{lem:objects}, we have $F(t(f_i))=F(s(f_{i+1}))$. 
In order to prove that the definition of $\theta F$ does not depend on the choice of representatives, it suffices 
to show that  $F(\sigma)F(\tau) =F(\mu)$ if $\sigma\tau\stackrel{c}\sim \mu$ and $F(\sigma)=F(\mu)$ if $\sigma \stackrel{ob}{\sim} \mu$. 
In fact, the equivalence relation $\stackrel{1}{\sim}$ is generated by 
the relations $\stackrel{c}\sim$ and $\stackrel{ob}{\sim}$. 

Suppose that $\sigma\tau\stackrel{c}{\sim} \mu $.  Then by definition, there exist 
composable morphisms $k \in \sigma$ and $h \in \tau$ such that $kh$ is in $\mu$ .  
Thus it follows that 
$(\theta F)(\sigma\tau)=(\theta F)(\sigma)(\theta F)(\tau) = F(k)F(h) =F(kh) = (\theta F)(\mu)$. 
If $\sigma \stackrel{ob}{\sim} \mu$, then $(\theta F)(\sigma)=F(id_a)$ and $(\theta F)(\mu)=F(id_b)$ for some 
$id_a \in \sigma$ and $id_b \in \mu$ with $a \stackrel{0}{\sim} b$. 
By Lemma \ref{lem:objects}, we see that $F(a) = F(b)$. This implies that 
$F(id_a)= id_{F(a)} = id_{F(b)} = F(id_b)$. 

For a morphism $\eta : F \to G$ in $\natural\mathsf{Sets}^{(\C, S)}$, we define 
$\theta(\eta) : \theta F \to \theta G$ by  $\theta(\eta)_{[x]}= \eta_{x}$. Observe that by definition, 
$\eta_x =\eta_y$ if $x \stackrel{0}{\sim} y$. 
We prove that for a morphism $\sigma_n\cdots \sigma_1 : [x] \to [y]$, the diagram
$$
\xymatrix@C35pt@R25pt{
(\theta F)_{[x]} \ar[d]_{(\theta F)(\sigma_n\cdots \sigma_1)}\ar[r]^-{\theta \eta_{[x]}} & (\theta G)_{[x]} 
\ar[d]^{(\theta G)(\sigma_n\cdots \sigma_1)} \\
(\theta F)_{[y]}  \ar[r]^-{\theta \eta_{[y]}} & (\theta G)_{[y]} 
}
$$
is commutative. To this end, it suffices to verify the fact for the case where $n=1$.  
Then the diagram is nothing but the commutative diagram which shows $\eta$ is a natural transformation from $F$ to $G$.   
It follows that $\theta$ is a well-defined functor. 

We recall the functor in Lemma \ref{lem:pi}. Since the functor preserves the partition, it follows that $\pi$ induces a functor 
$\pi^* : \mathsf{Sets}^{c[(\C, S)]} \to \mathsf{Sets}^{(\C, S)}$ and 
$\pi^*$ factors through the category 
$\natural\mathsf{Sets}^{(\C, S)}$.  
It is readily seen that $\pi^*$ is the inverse to $\theta$. 
We have the result. 
\end{proof}

\begin{proof}[Proof of Theorem \ref{thm:main}]
In the category $\mathsf{Sets}$, axioms  (G1)--(G3) are satisfied. Then so are in $\natural\mathsf{Sets}^{(\C, S)}$ because 
the colimits and pullbacks are constructed objectwise. In fact, let  $\eta : F \to G$ be a morphism in $\natural\mathsf{Sets}^{(\C, S)}$. 
Then, for a morphisms $f : x \to y$ and $g : x' \to y'$ in $(\C, S)$, we have commutative diagrams
$$
\xymatrix@C25pt@R20pt{
F(x) \ar[r]^{F(f)} \ar[d]_{\eta_x} & F(y) \ar[d]^{\eta_y} & \text{and} & F(x') \ar[r]^{F(g)} \ar[d]_{\eta_{x'}} & F(y') \ar[d]^{\eta_{y'}}\\
G(x) \ar[r]_{G(f)} & G(y) & &                                                               G(x') \ar[r]_{G(g)} & G(y'). 
}
$$
Suppose that  $f \sim_S g$.  Since $\eta$ is in the functor category $\natural\mathsf{Sets}^{(\C, S)}$, it follows that the diagrams coincide. Thus we see that products, coprocucts, 
pullbacks and pushouts are in the functor category. Hence arbitrary limits and colomits are in  $\natural\mathsf{Sets}^{(\C, S)}$. 

Theorem \ref{thm:module} implies that a set of generators in $\natural\mathsf{Sets}^{(\C, S)}$ is indeed induced by that of 
$\mathsf{Sets}^{c[(\C, S)]}$ via the isomorphism $\pi^*$; see Remark \ref{rem:generators} below.  
The assertion on the abelian group objects follows from \cite[8.13 Theorem]{J}. 

Since $\mathsf{Sets}^{\C}$ and $\mathsf{Sets}^{c[(\C, S)]}$ are classifying topoi, it follows from 
the argument in \cite[Chapter I, Section 2]{M} that the functor $\pi : \C \to c[(\C, S)]$ in Lemma \ref{lem:pi} induces a geometric morphism 
$(\pi^*, \pi_*) : \mathsf{Sets}^{\C} \to \mathsf{Sets}^{c[(\C, S)]}$; see \cite[page 12]{M} for an explicit form of the right adjoint $\pi_*$. 
Thus the proof of Theorem \ref{thm:module} allows us to obtain the diagram  
$$
\xymatrix@C35pt@R10pt{
\mathsf{Sets}^{\C \  \ }  \ar@/^5mm/[rr]^-{\pi_*}_-{\top}  & & \mathsf{Sets}^{c[(\C, S)]}  \ar[ll]^-{\pi^*}  \ar@/^4mm/[ld]^-{\pi^*}_-{\cong} \\
 & \natural\mathsf{Sets}^{(\C, S)} \ar[lu]^-\iota \ar[ru]^-{\theta} & 
}
$$
in which the inner triangle is commutative. Therefore, for objects $F$ in $\natural\mathsf{Sets}^{(\C, S)}$ 
and $G$ in $\mathsf{Sets}^{\C}$, we have natural bijections
\begin{align*}
\text{Hom}_{\mathsf{Sets}^{\C}}(\iota F, G) &= \text{Hom}_{\mathsf{Sets}^{\C}}(\pi^*\theta F, G) \\
&\cong  \text{Hom}_{\mathsf{Sets}^{c[(\C, S)]}}(\theta F, \pi_*G) \ \cong \ \text{Hom}_{\natural\mathsf{Sets}^{(\C, S)}}(F, \pi^*\pi_*G). 
\end{align*}
The last bijection is induced by the isomorphism in Theorem \ref{thm:module}. Since $\iota$ commutes finite limits immediately, it follows that 
$(\iota, \pi^*\pi_*)$ is a geometric morphism we require.  
This completes the proof. 
\end{proof}

\begin{rem}\label{rem:generators} Let $\C$ be a small category. The usual functor category $\mathsf{Sets}^{\C}$ has a set of generators 
$\{h_x\}_{x \in \C}$, where $h_x$ is the representation functor, namely, $h_x := \text{Hom}_{\C}( x , \  )$. In fact, given two distinct morphisms 
$f_1, f_2 : F \to G$ in $\mathsf{Sets}^{\C}$, there exists an object $x$ in $\C$ such that  $(f_1)_x \neq (f_2)_x$ as a map from  
$F(x)$ to $G(x)$. Thus  ${(f_1)}_x(u) \neq {(f_2)}_x(u)$ for some $u \in F(x)$. It follows from Yoneda Lemma that there exists a morphism 
$\alpha_u : h_x \to F$ such that ${(\alpha_u)}_x(id_x) = u$ and hence ${(f_i\circ \alpha_u)}_x(id_x) = {(f_i)}_x(u)$  for $i = 1, 2$. 
This implies that $f_1\circ \alpha_u \neq f_2\circ \alpha_u$. 

By virtue of Theorem \ref{thm:module}, we see that 
the set $\{\pi^*h_x\}_{x \in c[(\C, S)]}$ gives a set  of generators in $\natural\mathsf{Sets}^{(\C, S)}$ for each colored category $(\C, S)$. 
\end{rem}

\begin{rem}\label{rem:equivalences} Let $(\C, S)$ be a naturally colored category.  Suppose further that the set of objects in $c[(\C, S)]$ is finite. 
By Theorem \ref{thm:module} and Proposition \ref{prop:equality}, 
we have equivalences 
$$\mathsf{Mod}^{(\C, S)} = Ab(\mathsf{Sets}^{(\C, S)}) \simeq Ab(\mathsf{Sets}^{c[(\C, S)]}) = \mathsf{Mod}^{c[(\C, S)]} \simeq 
{\mathbb Z}[c[(\C, S)]]\text{-}\mathsf{Mod}, $$
where ${\mathbb Z}[c[(\C, S)]]\text{-}\mathsf{Mod}$ denotes the category of left ${\mathbb Z}[c[(\C, S)]]\text{-}$modules. 
Mitchell's embedding theorem gives the second equivalence. 
It turns out that $\mathsf{Mod}^{(\C, S)}$ has enough injectives and projectives. 
We observe that the second equivalence is also functorial if the schemoids satisfy the condition that $id_x \sim_S id_y$ for any objects $x$ and $y$ in $\C$. 
Indeed, for any morphism 
$u : (\C, S) \to (D, S')$ between such naturally colored categories, the functor 
$c[u] : c[(\C, S)] \to c[(D, S')]$ induced by $u$ gives rise to a functor 
$c[u]^* :  {\mathbb Z}[c[(\D, S')]]\text{-}\mathsf{Mod} \to {\mathbb Z}[c[(\C, S)]]\text{-}\mathsf{Mod}$ since 
$\sharp ob (c[(\D, S')]) = 1= \sharp ob (c[(\C, S)])$ and hence $c[u]$ is a monomorphism in the set of objects. 
\end{rem}

Let $\widetilde{S}(G)$ be a natural schemoid which a discrete group $G$ gives; 
see \cite[Example 2.6(iii)]{K-Matsuo} and the diagram (2.2). 
In view of Proposition \ref{prop:H} and Remark \ref{rem:equivalences}, we have an equivalence 
$\varphi : \mathsf{Mod}^{\widetilde{S}(G)} \stackrel{\simeq}{\longrightarrow}  {\mathbb Z}[G]\text{-Mod}$ 
of abelian categories which sends a constant sheaf $\underline{M}$ to the module $M$ with the trivial $G$-action. This implies the following result.

\begin{cor} One has an isomorphism 
$
H^*(\widetilde{S}(G),  A) \cong H^*(G, \varphi A) 
$
of abelian groups for any object $A$ in $\mathsf{Mod}^{\widetilde{S}(G)}$. 
\end{cor}

Let $\jmath H(n, 2)$ be the schemoid arising from the Hamming scheme $H(n, 2)$ of binary codes with length $n$; 
see \cite[Example 2.2]{K-Momose}.  
The result  \cite[Proposition 4.3]{K-Momose} 
yields that there exists an equivalence 
$\psi : \mathsf{Mod}^{\jmath H(n, 2)} \stackrel{\simeq}{\longrightarrow} {\mathbb Z}[{\mathbb Z}/2]\text{-Mod}$ 
of abelian categories which assigns to a constant sheaf $\underline{M}$ the module $M$ with the trivial ${\mathbb Z}/2$-action.  Thus we have 

\begin{cor} \label{cor:Hamming} For any $n \geq 1$, one has an isomorphism 
$
H^*(\jmath H(n, 2),  A) \cong H^*({\mathbb Z}/2, \psi A) 
$
of abelian groups for any object $A$ in $\mathsf{Mod}^{\jmath H(n, 2)}$. 
\end{cor}

Let $[C]$ be the category in Lemma \ref{lem:tame_qs} associated with a tame schemoid $(\C, S)$.  
Then the category algebra of $[C]$ is isomorphic to the Bose--Mesner algebra (\cite[page 111]{K-Matsuo}) of 
$(\C, S)$ if each structure constant is less than or equal to $1$; see \cite[Theorem 3.5]{K-Momose}. 
We see that the category $c[(\C, S)]$ in this article is a generalization of $[C]$. 
In fact, we have the following 
proposition. 

\begin{prop}\label{prop:tame_one}
The category $[C]$ is isomorphic to $c[(\C, S)]$ as a category if $(\C, S)$ is a tame schemoid. 
\end{prop}  

\begin{proof}
We first observe that $x \stackrel{0}{\sim} y$ if and only if $id_x \sim_S id_y$; 
see \cite[Remark 3.1]{K-Momose}. Thus we have $ob([\C])= ob(c[(\C, S)])$. 
Define a functor $F : [\C] \to c[(\C, S)]$ by $F([x])=[x]$ and $F(\sigma) = \sigma$ for 
$\sigma \in S$.  The well-definedness of $F$ follows immediately. 

We define a map  $G : M_{[x][y]} \to \text{Hom}_{[\C]}([x], [y])$ by 
$G(\sigma_n\cdots \sigma_1) = \sigma_n\cdots \sigma_1$ with the composite of morphisms in $[\C]$. 
Suppose that $\mu\tau   \stackrel{c}{\sim} \sigma $ for 
$\mu, \tau, \sigma \in S$, where 
$\stackrel{c}\sim$ is 
the relation mentioned when defining $c[(\C, S)]$. Then we have 
$G(\mu\tau )= \mu\tau  = \sigma  = G(\sigma )$. The second equality follows from 
the definition of the composite in $[\C]$. If $\sigma \stackrel{ob}{\sim}\mu$, then there exist  
$id_a \in \sigma$ and $id_b \in \mu$ such that $a \stackrel{0}{\sim} b$. 
Since $(\C, S)$ is tame, it follows that $id_a \sim_S id_b$ and hence $G(\sigma) = \sigma = \mu = G(\mu)$ in $[C]$. 
By definition, we see that $id_{[x]} = \sigma$ in $c[(\C, S)]$ if $id_x \in \sigma$. Therefore, we have 
$G(id_{[x]}) = G(\sigma) =\sigma = id_{[x]}$ in $[C]$. 
Thus $G$ induces a map 
$\widetilde{G} : \text{Hom}_{c[(\C, S)]}([x], [y]) \to \text{Hom}_{[\C]}([x], [y])$ and gives rise to a functor 
$\widetilde{G} : c[(\C, S)] \to [\C]$. It is readily seen that $\widetilde{G}$ is the inverse of the functor $F$. 
\end{proof}

We conclude this section with an example which shows that the category $\mathsf{Sets}^{(\C, S)}$ does {\it not} admit the topos structure, in general, 
with the same objectwise construction as in $\mathsf{Sets}^{\C}$. 

\begin{ex}\label{ex:pullbacks}
Let $\C$ be the small category defined by $ob(\C) = \{x, y\}$ and $mor(\C) =\{ id_x, id_y, y \stackrel{f}{\to} y \}$ with $ff=id_y$. 
We define a partition $S$ of  $mor(\C)$ by $S =\{\sigma, \tau\}$, where $\sigma=\{id_x, f\}$ and $\tau=\{id_y\}$. Then 
$(\C, S)$ is a colored category but neither a naturally colored one nor a schemoid.
This example tells us that the pullback in $\mathsf{Sets}^{\C}$ is not necessarily that in 
$\mathsf{Sets}^{(\C, S)}$. In fact, if $F \in \mathsf{Sets}^{(\C, S)}$, then $F(id_x) = F(f)$ and hence 
$
F(x) = F(y), F(id_x) = F(id_y) = F(f) 
$
as in Lemma \ref{lem:objects}. Since $id_x \not\sim_S id_y$, it follows that every natural transformation is locally constant.  

Let $U$ be the set $\{1, 2, 3\}$ and $F \in \mathsf{Sets}^{(\C, S)}$ the functor defined by 
$$
F(x) = F(y) = U, F(id_x) = F(id_y) = F(f) = id_U.
$$
Moreover, we define two natural transformations $\eta, \lambda : F \to F$ by 
$$
\eta_x(1) = \eta_x(2) = 1, \eta_x(3) = 3, \eta_y = id_U, 
\lambda_x(1) = \lambda_x(2) = 2,  \lambda_x(3) = 3, \lambda_y = id_U.
$$
Consider the pullback $F\times_FF$ of the diagram $\xymatrix@C20pt@R25pt{F \ar[r]^\lambda &F & F \ar[l]_{\eta}}$ 
in $\mathsf{Sets}^{\C}$. Then we have 
\begin{align*}
(F\times_FF)(id_x) &: (F\times_FF)(x) = \{3\} \to  (F\times_FF)(x)=\{3\} \ \  \text{and} \\
(F\times_FF)(f) &: (F\times_FF)(y) = U \to  (F\times_FF)(y)=U, 
\end{align*} 
which are distinct maps although both $id_x$ and $f$ are in $\sigma$. Observe that $\eta$ is not a morphism in $\natural\mathsf{Sets}^{(\C, S)}$ because 
$x \stackrel{0}{\sim} y$ 
but 
$\eta_x \neq \eta_y$.  
\end{ex}

\section{Geometric morphisms and computational examples}

With the general theory of topoi, we consider a way for computing cohomology of colored categories.  

Let $u : (\C, S) \to (\D, S')$ be a morphism of colored categories. 
Then we see that the functor $u^* : \natural\mathsf{Sets}^{(\D, S')} \to \natural\mathsf{Sets}^{(\C, S)}$ induced by 
$u$ preserves finite limits and arbitrary colimits. 
By virtue of the Special Adjoint Functor Theorem (\cite[page 129]{Mac}), we have a right adjoint $u_*$ to the functor $u^*$; 
see also \cite[0.46 Corollary and 7.13 Proposition]{J}.  This gives a geometric morphism 
$(u^*, u_*) : \natural\mathsf{Sets}^{(\C, S)} \to \natural\mathsf{Sets}^{(\D, S')}$ of topoi. As a consequence,  
we have the Leray spectral sequence $\{E_r^{*,*}, d_r\}$ with 
$$
E_2^{p,q} \cong H^p((\D, S'), (R^qu_*)(M))
$$
converging to $H^*((\C, S), M)$ with coefficients in an object 
$M$ of $\mathsf{Mod}^{(\C, S)}$; see \cite[8.17 Proposition]{J}. 
Moreover, for an abelian object $B$ in $\mathsf{Sets}^{(\D, S')}$, one has a homomorphism 
$u^* : H^*((\D, S'),  B) \to H^*((\C, S),  u^*B)$; see \cite[8.17 Proposition(i)]{J}.  
Furthermore, Kan extensions enable us to obtain adjoint functors 
$$
\xymatrix@C35pt@R25pt{
\mathsf{Sets}^{c[(\D, S')]} \ar[r]_-{c[u]^*}
& \mathsf{Sets}^{c[(\C, S)]}  \ar@/^7mm/[l]^-{\text{Ran}_{u}}_-{\perp} 
 \ar@/_6mm/[l]_-{\text{Lan}_{u}}^-{\perp} \\
\mathsf{Sets}^{(\D, S')} \ar[r]_-{u^*} \ar[u]^{\theta}_{\cong}
& \mathsf{Sets}^{(\C, S)}. \ar[u]_{\theta}^{\cong}
}
$$
if $(\C, S)$ and $(\D, S')$ are naturally colored categories.  
Observe that $\theta$ is an isomorphism in Theorem \ref{thm:module} and that the square is commutative. 

\begin{rem}
The relative schemoid cohomology of $u :  (\C, S) \to (\D, S')$ defined in \cite[Definition 2.7]{K-Momose} is indeed the submodule $E_2^{*, 0}$ in the $E_2$-term 
of the Leray spectral sequence mentioned above. 
\end{rem}

We give a computational example of cohomology of a colored category by using the Leray spectral sequence.  

\begin{ex}\label{ex:computaiton}
Let $(\C, S)$ be a colored subcategory of a schemoid $(P(K), S_K)$ associated with an abstract simplicial complex $K$; that is, 
$\C$ is a subposet of the face poset $P(K)$ of $K$ and $S=\{\widetilde{\sigma} \cap mor \C \mid \widetilde{\sigma} \in S_K\}$. Observe that 
the partition $S_K$ of morphisms of the poset $P(K)$ 
consists of sets $\widetilde{\sigma} := \{ \tau \to \nu \mid \nu\backslash\tau = \sigma \}$ for $\sigma \in K\cup\{\varnothing \}$; 
see \cite[Lemma A.1]{K-Momose} and the discussion before \cite[Remark A.3]{K-Momose} for the schemoid $(P(K), S_K)$. 

Let $({\mathcal N}, len)$ be the schemoid whose underlying category ${\mathcal N}$ 
consists of non-negative integers as objects 
and the one arrow $i \to j$ for objects $i$ and $j$ with $i\leq j$. 
The length of the arrow $i \to j$ is defined to be the difference $j -i$.  
Then the length gives the partition $len$ of $mor{\mathcal N}$; see \cite[Example 4.2]{K-Momose}. 
It is readily seen that $({\mathcal N}, len)$ is natural. 
By {\it collapsing} the Hasse diagram of the poset $\C$ horizontally, we have a morphism 
$u : (\C, S) \to ({\mathcal N}, len)$ of colored categories; see \cite[Remark A.2]{K-Momose} for more details. 
Let $M$ be an object in $\mathsf{Mod}^{(\C, S)}$. Consider the Leray spectral sequence $\{E_r^{*,*}, d_r\}$ converging to 
$H^*((\C, S), M)$. We have the composite 
$$
\Phi : \mathsf{Mod}^{({\mathcal N}, len)} \stackrel{\simeq}{\longrightarrow} \mathsf{Mod}^{c[({\mathcal N}, len)]} 
\stackrel{\simeq}{\longrightarrow} {\mathbb Z}[\sigma]\text{-}\mathsf{Mod}
$$
of equivalences of categories. The first equivalence follows from Theorem \ref{thm:module} 
and Mitchell's embedding theorem gives the second one. It follows from Remark \ref{rem:monoids}  that 
$c[({\mathcal N}, len)]$ is the free category generated by a endomorphism $\sigma$ with only one object.  
This yields isomorphisms 
\begin{align*}
E_2^{p, q} &\cong  H^p(({\mathcal N}, len), (R^qu_*)(M)) \\
&\cong \text{Ext}^p_{{\mathbb Z}[\sigma]}({\mathbb Z}, \Phi (R^qu_*)(M))\\
&\cong H^p(\text{Hom}(\wedge(\tau), \Phi (R^qu_*)(M)); \delta), 
\end{align*}
where the differential $\delta$ is defined by $\delta(f)(\tau) =  \sigma f(1)$. 
The Koszul resolution $({\mathbb Z}[\sigma]\otimes \wedge(\tau), d) \to {\mathbb Z}$ 
of ${\mathbb Z}$ as a ${\mathbb Z}[\sigma]$-module, in which $d(\tau) = \sigma$ and $\deg \tau = -1$, 
gives rise to the last isomorphism. There is no element with degree less than or equal to $-2$ in the Koszul resolution.  
Thus, we see that $E_2^{p,*}=0$ if $p\geq 2$ and hence all differentials in the $E_2$-term are trivial. 
Since the spectral sequence collapses at the $E_2$-term, it follows that $E_2^{*, *}\cong E_\infty^{*, *}$.  It turns out that 
$
E_2^{p,q} \cong F^pH^{p+q}((\C, S), M)/F^{p+1}H^{p+q}((\C, S), M).  
$
\end{ex}

\begin{rem}
We consider again the schemoid $(P(K), S_K)$ arising from an abstract simplicial complex $K$ with $K_0$ the set of vertices.  
Since $(P(K), S_K)$ is natural, it follows from Remark \ref{rem:monoids} and the proof of \cite[Lemma A.1]{K-Momose} that 
the small category $c[(P(K), S_K)]$ is a monoid of the form 
$$
\langle \widetilde{\sigma} \in S_K \mid \widetilde{\tau}\widetilde{\mu} =\widetilde{\sigma}  \ \text{if} \  p_{\widetilde{\tau}\widetilde{\mu}}^{\widetilde{\sigma}} \neq 0 \rangle, 
 \ \text{where} \ \ 
 p^{\widetilde{\sigma}}_{\widetilde{\tau} \widetilde{\mu}}= \begin{cases}
1 & \text{if $\sigma =  \tau \sqcup \mu$}  \\
0 & \text{otherwise}. 
\end{cases}
$$
Thus we see that the category algebra ${\mathbb Z}c[(P(K), S_K)]$ is isomorphic to the algebra $R_K:= T(x_i)/(x_ix_j - x_jx_i \mid \{i, j\} \in K)$, 
where $T(x_i)$ denotes the tensor algebra over ${\mathbb Z}$ generated by elements $x_i$ with indexes in $K_0$. 
In fact, a homomorphism $\varphi : {\mathbb Z}\langle \widetilde{\sigma} \in S_K\rangle \to R_K$ of algebras 
can be defined by $\varphi(\widetilde{\sigma}) = x_{i_1}\cdots x_{i_l}$ if $\sigma = \{i_1 , ..., i_l \}$. Moreover, we define a homomorphism 
$\psi : T(x_i) \to {\mathbb Z}c[(P(K), S_K)]$ of algebras by $\psi(x_i) = \widetilde{\{i\}}$. It follows that $\varphi$ induces an isomorphism 
from ${\mathbb Z}c[(P(K), S_K)]$ to $R_K$ with the inverse given by $\psi$. The same argument yields that the monoid 
${\mathbb Z}c[(P(K), S_K)]$ is isomorphic to a monoid of the form 
${\mathcal M}: = \langle \{i \} \in K_0 \mid \{i\}\{j\} =\{j\}\{i\} \ \text{if} \  \{i, j\} \in K \rangle.$ 
In consequence, the discussion in Remark \ref{rem:equivalences} allows us to obtain an isomorphism 
$\mathsf{Sets}^{(P(K), S_K)} \cong \mathsf{Sets}^{\mathcal M}$ of topoi and an equivalence 
$\mathsf{Mod}^{(P(K), S_K)} \simeq R_K\text{-}\mathsf{Mod}$ of abelian categories. 

An unsatisfactory feature of the result is that the category 
$\mathsf{Mod}^{(P(K), S_K)}$ depends only on the $0$ and $1$-simplices of $K$. Indeed, for simplicial complexes $K$ and $K'$, 
the algebras $R_K$ and $R_{K'}$ are isomorphic if the complex of $1$-skeletons of $K$ 
is isomorphic to that of $K'$. However, the result \cite[Proposition A.5]{K-Momose} asserts that the Bose-Mesner algebra of $(P(K), S_K)$  
is isomorphic to the {\it squrefree} Stanley-Reisner ring. 
It will be expected that the study of the schemoid $(P(K), S_K)$ is developed with the Bose-Mesner algebra. 
\end{rem}

We here recall that the $q$-ary Hamming scheme $H(n, q)$ with length $n$ consists of the set $X=\{0, ..., q-1\}^n$ and the partition 
$S_H=\{\sigma_l\}_{0\leq l  \leq n}$ of $X\times X$, where $n\geq 1$ and 
$\sigma_l = \{((x_1, ..., x_n), (y_1, ..., y_n)) \mid \sharp \{i \mid x_i\neq y_i\}= l \}. $
The following result asserts that the cohomology of a schemoid, 
which has a particular morphism to $H(n, 2)$ for some $n$,  is non-vanishing everywhere.   

\begin{prop} \label{prop:App} Let $(\C, S)$ be a colored category. Suppose that $n \geq 1$ and that there exist a morphism 
$u : (\C, S) \to  \jmath H(n, 2)$ of colored categories and an element $\tau$ in the partition $S$ such that 
 $u(\tau) \subset \sigma_{2m+1}$ for some $m$ and 
$\tau$ has an invertible morphism in itself; that is, there exists an invertible morphism $f$ in $\tau$ such that $f^{-1}$ is also 
in $\tau$. 
Then one has an epimorphism  from the cohomology 
$H^*((\C, S), \underline{{\mathbb Z}/2})$ to the group cohomology $H^*({\mathbb Z}/2, {\mathbb Z}/2)$. In particular,  
$H^k((\C, S), \underline{{\mathbb Z}/2}) \neq 0$ for any $k$. 
\end{prop}

\begin{proof}
Let $f$ be an invertible morphism in $\tau$. 
Then we have a well-defined morphism $v : \widetilde{S}({\mathbb Z}/2) \to (\C, S)$ of schemoids 
with $v(0 \to 1) = f$. Since $u(f) \in \sigma_{2m+1}$, it follows from the proof of \cite[Proposition 4.3]{K-Momose} that the composite 
$u\circ v :  \widetilde{S}({\mathbb Z}/2) \to (\C, S) \to \jmath H(n, 2)$ of functors gives rise to 
a Morita equivalence between $\widetilde{S}({\mathbb Z}/2)$ and the Hamming scheme; that is, 
one has the equivalence $(u\circ v)^* :  \mathsf{Mod}^{\jmath H(n, 2)} \stackrel{\simeq}{\to} 
\mathsf{Mod}^{\widetilde{S}({\mathbb Z}/2)} 
\simeq {\mathbb Z}[{\mathbb Z}/2]\text{-}\mathsf{Mod}$ of abelian categories.  
Therefore, the composite 
$$
\xymatrix@C20pt@R25pt{
H^*(\jmath H(n, 2), \underline{{\mathbb Z}/2}) \ar[r]^-{u^*} & H^*((\C, S), 
\underline{{\mathbb Z}/2}) \ar[r]^-{v^*} & H^*({\mathbb Z}/2, {\mathbb Z}/2)
}
$$
is an isomorphism.
Observe that the inverse image functor $u^*$ sends the constant sheaf $\underline{{\mathbb Z}/2}$ to itself. 
This implies that $u^*$ is a monomorphism. We have the result. 
\end{proof}

We apply Proposition \ref{prop:App} to a very small colored category. 

\begin{ex} Let $(\C, S)$ be the colored category defined by the left-hand side diagram below. 
$$
\xymatrix@C25pt@R20pt{
(\C, S) : \!\!\!\!\!\! &{00}  \ar@{<->}[r] \ar@{~>}[d] &  {01}, \ar@{.>}[ld] & H(2,2): \!\!\!\! & {00}  \ar@{<->}[r] \ar@{<.>}[rd]  \ar@{<->}[d] &  {01}  \ar@{<->}[d] \\
             & {10}                                         &                            &                 & {10} \ar@{<->}[r] \ar@{<.>}[ru] &  {11}
}
$$
Observe that $(\C, S)$ is not a schemoid. 
We define a partition-preserving functor $u : (\C, S) \to H(2,2)$ by $u(ij) =ij$. 
Since $u$ sends the partition, which contains $00 \to 01$ and hence also $01 \to 00$,  to 
$\sigma_1$, it follows that the functor satisfies the condition in Proposition \ref{prop:App}. Then $H^k((\C, S), \underline{{\mathbb Z}/2}) \neq 0$ for any $k$. 
On the other hand, 
the underlying small category $\C$ of the discrete schemoids $\calK U(\C, S)$ has the terminal object. Thus we see that 
$H^0(\calK U(\C, S), \underline{{\mathbb Z}/2}) = {\mathbb Z}/2$ and $H^k(\calK U(\C, S), \underline{{\mathbb Z}/2}) = 0$ for $k > 0$. In fact, 
the classifying space of the category $\calK U(\C, S)$ is contractible. 
\end{ex}

Remark \ref{rem:equivalences} asserts that the problem on the Morita equivalences between natural schemoids $(\C, S)$ 
is reduced to that between the category algebras of the associated small categories $c[(\C, S)]$; 
see \cite[Section 2]{K-Momose} for the Morita equivalence between schemoids. 
We here address such small categories associated with the Hamming schemes and the Johnson schemes; 
see also Proposition \ref{prop:H}. 

\begin{rem}\label{rem:order}
Suppose that $(X, S)$ is symmetric; that is, for any element $\sigma$ in $S$, we have 
$\sigma = \sigma^*$, where  $\sigma^*:= \{(x, y) \in X\times X \mid (y, x) \in \sigma\}$. Then the order of 
each element of $c[\jmath(X, S)]$ is at most two. 
\end{rem}

\begin{cor}\label{cor:trivial_one}
\text{\em (}cf. \cite[Proposition 4.2]{Numata_2}\text{\em )} Let $l$ and $n$ be integers greater than or equal to $1$ and $q$ an integer greater than $2$. 
Then, the group $c[\jmath H(l, q)]$ is trivial while $c[\jmath H(n, 2)] = {\mathbb Z}/2$. 
In consequence, there is no non-trivial morphism $u$ of schemoids from $\jmath H(l, q)$ to $\jmath H(n, 2)$ 
such that $u(\sigma_s) \subset \sigma_{2m+1}$ for some $s$ and $m$. 
\end{cor}

\begin{proof} 
For any positive integers $m$ and $r$, 
the Hamming scheme $\jmath H(m, r)$ is symmetric.  Moreover,  
Remarks \ref{rem:monoids} and \ref{rem:order} allow us to deduce that $c[\jmath H(m, r)]$ 
is a group generated by the single element $\sigma_1$ with order at most $2$. In particular, 
we see that $c[\jmath H(n, 2)] = {\mathbb Z}/2$ for any $n$.  In fact, if $c[\jmath H(n, 2)]$ is trivial, then the equivalence 
$ \mathsf{Mod}^{\jmath H(n, 2)} \simeq {\mathbb Z}[c[\jmath H(n, 2)]]\text{-}\mathsf{Mod} \simeq  {\mathbb Z}\text{-}\mathsf{Mod}$ 
mentioned in Remark \ref{rem:equivalences} deduces that $H^*(\jmath H(n, 2); \underline{{\mathbb Z}/2})$ is acyclic. 
However, Corollary \ref{cor:Hamming} implies that the cohomology is not acyclic, which is a contradiction.   

Consider the Hamming scheme $H(l, q)$ with $q > 2$. We have a sequence of morphisms 
$(0, 0, 0, ..., 0) \stackrel{f_1}{\longrightarrow} (1, 0, 0, ..., 0) \stackrel{f_2}{\longrightarrow} (2, 0, 0, ..., 0)$
in the underlying category of $\jmath H(l, q)$ for which each $f_i$ is in $\sigma_1$. Moreover, it is readily seen that the composite $f_2 \circ f_1$ is also in 
$\sigma_1$ and hence $\sigma_1^2 = \sigma_1$ in the group $c[\jmath H(l, q)]$. 
Since the order of $\sigma_1$ is at most $2$, it follows that 
$c[\jmath H(l, q)]$ is trivial. 

Suppose that there exists a non-trivial morphism $u : \jmath H(l, q) \to \jmath H(n, 2)$ of schemoids which satisfies the condition mentioned in the assertion. 
Then the homomorphism $c[u] : c[\jmath H(l, q)] \to c[\jmath H(n, 2)]$ induced by $u$ sends $\sigma_1$ to $\sigma_1$. 
This yields that $\sigma_1 =1$ in $c[\jmath H(n, 2)]$, which is a contradiction. 
\end{proof}

\begin{rem} For an integer $q > 2$, 
we may have a non-trivial morphism $u : \jmath H(l, q) \to \jmath H(n, 2)$ with $\sigma_s \subset \sigma_{2m}$ for some $s$ and $m$. For example, 
consider a morphism $u : \jmath H(1, 3) \to \jmath H(3, 2)$ of schemoids which is defined by $u(0) = 010$, $u(1) = 001$ and $u(2) = 111$. Then it is readily seen that 
$u(\sigma_1) \subset \sigma_2$.  
\end{rem}

\begin{rem}
Let $V$ be a finite set of $v$ elements  and $d$ an integer with $d\leq \frac{v}{2}$. Then we have the Johnson scheme $J(v, d)=(X, S)$, where 
$X = \{x \mid x \subset V, \sharp x = d\}$ $R_i= \{(x, y) \in X\times X \mid \sharp(x\cap y) = d-i \}$ and $S$ is the partition of $X\times X$ 
consisting of the sets $R_i$ for $i = 0, ...,d$. 
The same argument as in the proof of Corollary \ref{cor:trivial_one} enables one to conclude that 
$c[\jmath J(v, d)]$ is trivial if $(v, d)\neq (2,1)$ and that $c[\jmath J(2, 1)]= {\mathbb Z}/2$.  
\end{rem}

\section{The standard representation of an association scheme and a spectral sequence} 

The standard representation of an association scheme (abbreviated AS henceforward) on the Bose--Mesner algebra of the AS    
plays an important role in representation theory of ASs. 
In fact, there exist ASs for which we obtain an isomorphism in the Bose--Mesner algebras but not in the standard representations.  
Therefore, one might expect that the representations reflect combinatorial data which ASs have. 
We indeed observe such a phenomenon in modular representation theory of ASs; see \cite[\S 5]{H-Y}.

Let $(X, S_X)$ be an AS of order $n$ and ${\mathbb Z}(X, S_X)$ the Bose--Mesner algebra of $(X, S_X)$. 
Then we have the standard representation ${\mathbb Z}(X, S_X) \to M_X$, namely the inclusion to the $n\times n $ matrix algebra with integer coefficients.   
The extension functor $\text{Ext}_{{\mathbb Z}(X, S_X)}^{*}({\mathbb Z}, M_X)$ is an invariant for isomorphisms of ASs. More precisely, we have 

\begin{prop} Let $\varphi : (X, S_X) \to (Y, S_Y)$ be an isomorphism of ASs; see \cite[1.7]{Z_book1}. Then $\varphi$ induces an isomorphism 
$\text{\em Ext}_{{\mathbb Z}(X, S_X)}^*({\mathbb Z}, M_X) \cong \text{\em Ext}_{{\mathbb Z}(Y, S_Y)}^*({\mathbb Z}, M_Y).
$
\end{prop}

\begin{proof} 
The definition of an isomorphism of ASs gives a commutative diagram
$$
\xymatrix@C20pt@R18pt{
M_X \ar[r]^-{\widetilde{\varphi}}_-{\cong} & M_Y \\
{\mathbb Z}(X, S_X) \ar[u] \ar[r]_-{\varphi}^-{\cong} & {\mathbb Z}(Y, S_Y), \ar[u] }
$$
where vertical maps are inclusions. The lower horizontal isomorphism induces an isomorphism 
$\varphi^* : {\mathbb Z}(Y, S_Y)\text{-}\mathsf{Mod} \stackrel{\cong}{\longrightarrow} {\mathbb Z}(X, S_X)\text{-}\mathsf{Mod}$. 
Thus, we have a diagram 
$$
\xymatrix@C35pt@R15pt{
\text{Ext}_{{\mathbb Z}(Y, S_Y)}^*({\mathbb Z}, M_Y) \ar[r]^-{\varphi^*}_-\cong & \text{Ext}_{{\mathbb Z}(X, S_X)}^*({\mathbb Z}, \varphi^*M_Y) \ar@{=}[d] \\
\text{Ext}_{{\mathbb Z}(X, S_X)}^*({\mathbb Z}, M_X) \ar[r]_-{\text{Ext}(1, \widetilde{\varphi})}^-{\cong} &
\text{Ext}_{{\mathbb Z}(X, S_X)}^*({\mathbb Z}, \widetilde{\varphi}(M_X)) 
}
$$
of isomorphisms between the extension functors. 
The vertical equality follows from the commutative diagram above. This completes the proof.  
\end{proof}

The cohomology of a colored category introduced in Definition \ref{defn:cohomology} relates to the invariant above via 
the Leray spectral sequence.  
Let $u : \jmath(X, S_X) \to (\C, S)$ be a morphism of colored categories and $\pi : (X, S_X) \to Quo(S_X)$ the quotient map, 
where $Quo(S_X)$ stands for the factor scheme by the thin residue of $(X, S_X)$; see Proposition \ref{prop:H} and Appendix A. 
We observe that $\pi$ is an admissible map and hence it induces a morphism 
$\pi : {\mathbb Z}(X, S_X) \to {\mathbb Z}Quo(S_X)$ of algebras; see \cite[Lemma 3.9, Corollary 6.4]{F}. 
Then, we have a diagram of adjoint functors. 
$$
\xymatrix@C20pt@R20pt{
\natural\mathsf{Mod}^{(\C, S)} \ar[r]^-{u^*}
& \mathsf{Mod}^{\jmath(X, S_X)}  \ar@/^6mm/[l]^-{u_*}_-{\perp} \ar[r]^-{\simeq}&  {\mathbb Z}Quo(S_X)\text{-}\mathsf{Mod} 
\ar[r]^-{\pi^*} & {\mathbb Z}(X, S_X)\text{-}\mathsf{Mod}. 
\ar@/_6mm/[l]_-{\pi_!}^-{\perp} 
 \ar@/^6mm/[l]^-{\pi_*}_-{\perp} 
}
 \eqnlabel{add-2}
$$
Here the equivalence in the middle is obtained by Proposition \ref{prop:H} and Remark \ref{rem:equivalences}; see also Section 3. 

\begin{thm}\label{thm:LSS} Under the setting above, one has a first quadrant spectral sequence converging to 
$\text{\em Ext}_{{\mathbb Z}(X, S_X)}^*({\mathbb Z}, M)$ with 
$
E_2^{p,q} \cong H^p((\C, S), (R^q u_*\pi_*)(M)),
$
where $M$ is a left ${\mathbb Z}(X, S_X)$\text{-}module. 
\end{thm}

\begin{proof}
The pair of tensor-hom adjoints gives an equivalence $\text{Hom}_{{\mathbb Z}Quo(S_X)}({\mathbb Z}, \  )\circ \pi_* \cong \text{Hom}_{{\mathbb Z}(X, S_X)}({\mathbb Z}, \ )$
of functors. 
The same argument as in the proof of \cite[8.17 Proposition]{J} allows us to conclude that 
$\text{Hom}_{\mathsf{Mod}^{(\C, S)}}(\underline{{\mathbb Z}}, \  )\circ u_* \cong \text{Hom}_{\mathsf{Mod}^{\jmath(X, S_X)}}(\underline{{\mathbb Z}}, \ )$. 
Since $u^*$ and $\pi^*$ are exact, it follows from \cite[2.8.9 Proposition]{S} that the right adjoint $u_*\pi_*$ preserves injective objects. Then the Grothendieck spectral sequence gives the one in the assertion. 
\end{proof}

\begin{rem}
Let $(X, S)$ be a coherent configuration. Then the map $\pi : \jmath (X, S) \to c[\jmath (X, S)]$ defined in the proof of Theorem \ref{thm:module} is admissible; see 
\cite[Definition 6.2]{K-Matsuo}.  By virtue of \cite[Proposition 6.7]{K-Matsuo}, 
we see that the group $Quo(S_X)$ is replaceable with the groupoid $c[\jmath (X, S)]$ in the diagram (4.1). 
Observe that $\jmath (X, S)$ is a natural schemoid; see Remark \ref{rem:mildness} (ii). 
Thus, Theorem \ref{thm:LSS} remains valid after replacing the AS to a more general coherent configuration. 
\end{rem}

As remarked in \cite[Remark 6.6]{F}, in general, 
a morphism $u : (X, S_X) \to (Y, S_Y)$ of ASs does {\it not} induce a morphism of algebras between the Bose--Mesner 
algebras naturally. However, the morphism $u$ induces a group homomorphism $Quo(S_X) \to Quo(S_Y)$; 
see Proposition \ref{prop:H}. Thus we have a variant of Theorem \ref{thm:LSS}. 

\begin{cor}\label{cor:Change-of-rings} Let $u : (X, S_X) \to (Y, S_Y)$ be a morphism of ASs and $M$ a left ${\mathbb Z}(X, S_X)$\text{-}module. Then, 
there exists a first quadrant spectral sequence converging to $\text{\em Ext}_{{\mathbb Z}(X, S_X)}^*({\mathbb Z}, M)$ with 
$$
E_2^{p,q} \cong H^p(Quo(S_Y), (R^q u_*\pi_*)(M)) 
= \text{\em Ext}_{{\mathbb Z}Quo(S_Y)}^p({\mathbb Z}, \text{\em Ext}_{{\mathbb Z}(X, S_X)}^q({\mathbb Z}Quo(S_Y), M) ). 
$$
\end{cor}

The spectral sequence above is nothing but the change-of-rings spectral sequence. 

\begin{ex} Let $(X, S)$ be the Hamming scheme $H(n, 2)$ of binary codes and $M$ a left ${\mathbb Z}H(n, 2)$\text{-}module. 
By virtue of Proposition \ref{prop:H} and the proof of Corollary \ref{cor:trivial_one}, we have 

\begin{assertion}\label{assertion:Hamming_scheme} 
$Quo(S) \cong {\mathbb Z}/2$. 
\end{assertion}

Let $\{E_r^{*, *}, d_r\}$ be the spectral sequence in Corollary \ref{cor:Change-of-rings} for the identity on $(X, S)$. 
Then, the periodic resolution of ${\mathbb Z}$ deduces that 
$$
E_2^{1,0} \cong \text{Ker}((1+t) : \widetilde{M} \to \widetilde{M})/ (t-1)\widetilde{M}, 
$$ 
where 
$\widetilde{M}$ denotes the left ${\mathbb Z}({\mathbb Z}/2)$-module 
$\text{Hom}_{{\mathbb Z}H(n, 2)}({\mathbb Z}({\mathbb Z}/2), M)$ and 
$t$ is the generator of ${\mathbb Z}/2$. 
Each element in $E_2^{1,0}$ is a permanent cycle. 
Then, we see that $E_2^{1,0}$ is a submodule of $\text{Ext}_{{\mathbb Z}H(n, 2)}^1({\mathbb Z}, M)$. 
\end{ex}


We conclude this article with observations and expectations for our work.

\begin{rem}\label{rem:assertions} (i) Thanks to Theorem \ref{thm:main}, we can define cohomology of colored categories;  
see Definition \ref{defn:cohomology}. 
Then, the cohomology is applicable to more general objects rather than schemoids. This means that important combinatorial data 
of a schemoid are ignored exactly in the definition. However, 
such data may characterize an algebraic property of the cohomology of a schemoid as a closed and orientable manifold 
possesses the Poincar\'e duality for the singular cohomology. \\
(ii) Theorems \ref{thm:main}, \ref{thm:module} and Proposition \ref{prop:H} 
enable us to be aware that an appropriate partition of morphisms in $\mathsf{Sets}$ is needed 
for giving {\it novel} cohomology for association schemes. The authors do not yet have a candidate of such a partition. 

On the other hand, instead of the category $\mathsf{Sets}$, if we consider $\mathsf{Cat}$ the category of small categories, 
then we have a functor category $\mathsf{Cat}^{(\C, S)}$ consisting of functors $\Theta : \C \to \mathsf{Cat}$ with 
$\Theta (f)\cong \Theta (g)$ if $f \sim_S g$; that is, there exists a natural isomorphism $\eta :  \Theta (f)\Rightarrow \Theta (g)$ in 
 $\mathsf{Cat}$. As for a morphism in $\mathsf{Cat}^{(\C, S)}$, it is defined to be a natural transformation $\nu : F \Rightarrow G$ which satisfies the condition that 
$\nu_x \cong \nu_y$ if $id_x \sim_S id_y$. 
We color morphisms of $\mathsf{Cat}$ so that morphisms $K$ and $H$ has the same color if there exists a natural isomorphism between them. An object in $\mathsf{Cat}^{(\C, S)}$ is regarded as a functor which preserves the coloring. 
 Thus, one might expect that the notion of a 2-topos in the sense of Street \cite{Street} is useful in the study of colored categories, 
 schemoids and association schemes. \\ 
(iii) Let $(\C, S)$ be a natural schemoid. In view of Proposition \ref{prop:H}, 
the monoid $c[(\C, S)]$ in Remark \ref{rem:monoids} may play an important role in representation theory 
for $(\C, S)$. Thus, in such consideration, it will be needed to introduce the notion of a {\it closed subset} in a schemoid; 
see \cite{Z} for a crucial role of the thin residue in the study of association schemes.   \\
(iv) We observe that a set of generators of the topos $\mathsf{Sets}^{(\C, S)}$ is described in the proof of Theorem \ref{thm:main}.  
Indeed, the generators are given by the representation functors in $\mathsf{Sets}^{c[(\C, S)]}$. 
Moreover, $\mathsf{Sets}^{(\C, S)}$ has {\it enough points}. An advantage of these facts is 
that we may use the generators to construct a topological space on which sheaves approximate 
$\mathsf{Sets}^{(\C, S)}$; see \cite{B-M, B-M2} and \cite[IV 1.1. Theorem]{M} 
for more details. 
One might expect to be able to develop homotopical algebra for colored categories 
with such sheaves.
\end{rem}

\noindent
{\it Acknowledgements.} 
The first author thanks Toshiki Aoki for discussing topoi, without which 
he could not relate schemoids to topoi. Authors are grateful to Akihide Hanaki 
for the precious discussion on representation theory of association schemes and for showing them Proposition \ref{prop:H}.  
They thank two referees for careful reading the previous version of this paper and giving valuable comments without which, 
in particular, it would not have been possible to revise  Theorem \ref{thm:main} and its proof correctly. 
In fact, Example \ref{ex:pullbacks} is due to one of the referees. 
A comment of another referee has allowed the authors to refine the main theorem with the geometric morphism. 

This research was partially supported by a Grant-in-Aid for Scientific
Research HOUGA JP16K13753 from Japan Society for the Promotion of Science.

\section{Appendix A}

In this section, we describe the proof of  Proposition \ref{prop:H}  due to Hanaki \cite{Hanaki_privatecom} 
for the reader. We use the same notation and terminology as in \cite{Z_book1}. 

We first recall the factor scheme $(X, S)^{{\bf O}^{\vartheta}(S)}$ of an association scheme $(X, S)$. 
Let $T$ be the thin residue ${\bf O}^{\vartheta}(S)$ of the association scheme $(X, S)$. 
Observe that $T$ is a subset of $S$ generated by $ss^*$ for $s\in S$: 
$T = \langle ss^* \!\mid \! s \in S \rangle$. Then, the result \cite[Theorem 2.3.4]{Z_book1} yields that 
the factor scheme $(X, S)^{{\bf O}^{\vartheta}(S)} = (X/T, S//T)$ is thin; that is, 
$S//T$ is a group with the complex product. 
By definition, an element in $X/T$ is a subset of $X$ of the form $xT = \cup_{\sigma \in T}x\sigma$, where 
$x\sigma = \{ y \in X \mid (x, y) \in \sigma \}$. Moreover, $S//T$ consists of elements $s^T := \{(xT, yT) \mid y \in xTsT \}$. 
Under the complex product, we have 
$s^Tt^T = u^T$ if $p_{st}^u \neq 0$.  Observe that $1^T = 1_{S//T}$ the unit of the group $S//T$ and $(s^*)^T = (s^T)^*$ the inverse of $s^T$. 

\begin{proof}[Proof of Proposition \ref{prop:H} {\em (}\cite{Hanaki_privatecom}{\em )}]  
Let $F$ be the free group generated by $S$ and $N$ the normal subgroup 
generated by elements of the form $\sigma\tau\mu^*$ with $p_{\sigma\tau}^\mu \neq 0$. 
This yields  
that the category $c[\jmath(X, S)]$ is isomorphic to the quotient group $F/N$; see Remark \ref{rem:monoids}. Thus,  
in order to prove the proposition, it suffices to show that $F/N \cong S//T$ as a group. 
We define a  map $h : F/N \to S//T$ by $f([\sigma]) = \sigma^T$. It is readily seen that $h$ is a well-defined epimorphism.  
For any $\sigma$ and $\tau$ in $S$, there exists $\mu \in S$ such that $p_{\sigma\tau}^\mu \neq 0$ 
and hence  $[\sigma][\tau] = [\mu]$ in $F/N$. This implies that 
each element $x$ in $F/N$ has a representative of the form $\mu$, where $\mu \in S$. 
Suppose that $[\mu]$ is in the kernel of $h$. Then $1^T = h([\mu]) = \mu^T$. 
It follows from \cite[Proposition 1.5.3]{Z_book1} that $\mu$ is in  $T$. Thus there exist 
$s_1, ..., s_r \in S$ such that $\mu \in s_1s_1^* \cdots s_rs_r^*$. This enables us to deduce that 
$$
[\mu] = [s_1][s_1^*] \cdots [s_r][s_r^*] = 1_{F/N}. 
$$
The first equality follows from the inductive use of the fact that $[\sigma][\tau] = [\mu]$ if $p_{\sigma\tau}^\mu \neq 0$. 
 It is immediate that $h$ is a functorial. We have the result. 
\end{proof}

\section{Appendix B}

In this section, we rewrite certain of the definitions described in the body of the manuscript 
with terminology which is used in \cite{Numata_2} for considering  
a colored category. Such descriptions may yield a topos taking the place of the functor category $\mathsf{Mod}^{(\C, S)}$ 
in the study of schemoids; see Remark \ref{rem:assertions} (ii).  

Let $\C$ be a small category consisting of sets $\C_0$ and $\C_1$ of objects and morphisms, respectively.  
Let $\eel : \C_1 \to I$ be a surjective map to a set $I$, which is called a {\it coloring map}. A colored category $(\C, S)$ is nothing but a category $\C$ endowed with 
a partition $S=\{ \eel^{(-1)}(i) \mid i\in I \}$ of $\C_1$ for some coloring map $\eel$. 
We then have a diagram consisting of solid arrows
$$
\xymatrix@C15pt@R15pt{
\C_1\times_{\C_0}\C_1 \ar[rr]^-{\circ} \ar[d]_{\eel \times \eel}&  & \C_1 \ar[d]^{\eel} \ar@<0.5ex>[rr]^-{s} \ar@<-0.5ex>[rr]_-{t} & &\C_0 \\
I \times I \ar@{..>}[r] & 2^I & I \ar[l]^-c\\
}
$$
in which the upper line is the usual diagram which explains the small category $\C$, where 
$c$ denotes the canonical injective map to the power set. We see that 
the colored category $(\C, S)$ is a schemoid if and only if for a triple $i, j, k \in  I$ and for any morphisms $f, g \in \eel^{-1}(i)$, one has a bijection 
$$\circ^{-1}(f) \cap (\eel \times \eel)^{-1}(j, k)\cong \circ^{-1}(g) \cap (\eel \times \eel)^{-1}(j, k). $$
The complex product of schemoid which is defined by the same way as in an association scheme gives the dots arrow; see \cite{Z_book1}  
and also  the comment after \cite[Lemma 6.3]{K-Matsuo}.  We observe that the square in the diagram above is commutative if the schemoid $(\C, S)$ is tame; see \cite[page 230]{K-Momose}. 


The definition of a natural colored category (Definition \ref{defn:mild_object}) is also rewritten  
with a coloring map.  
We first recall the Kronecker category or $2$-Kronecker quiver. 
The {\it Kronecker category} $1\xrightarrowrightarrow{\ttt}{\sss}0$,
denoted $\QQQ$,
is a category with two objects $1,0$ and with four morphisms
$\id_1,\id_0,\sss,\ttt$. 
Both of $\sss$ and $\ttt$ are
morphisms from $1$ to $0$.
A functor $G$ from $\QQQ$ to $\Sets$ 
can be identified with a digraph in the following manner:
$G(0)$ and $G(1)$ are the sets of vertices and edges, respectively.
Each $f\in G(1)$ is an edge from $G_\sss(f)$ to $G_\ttt(f)$.
Consider a small category $\C$ with a coloring map $\eel : \C_1 \to I=: I_1$.
Let $I_0$ be the set $\{\eel(\id_x) \mid  x\in\C_0 \}$.
We define a map $\eel_1 : \C_1 \to I_1$ by $\eel_1(f)=\eel(f)$.
We also define a map $\eel_0 : \C_0 \to I_0$ by $\eel_0(x)=\eel(\id_x)$.
Then it is a naturally colored category 
if and only if
there exist maps $\bar s$ and $\bar t$ from $I_1$ to $I_0$
such that 
\begin{align}
(\bar s\circ \eel_1)(f)&=(\eel_0\circ s)(f), \text{ and}
\label{nu:eq:diaram:s}
\\
 (\bar t\circ \eel_1)(f)&=(\eel_0\circ t)(f) 
\label{nu:eq:diaram:t}
\end{align}
for each morphism $f$ of $\C$.
In this case, 
we have a functor $G$ from $\QQQ$ to $\Sets$ defined by
$G(0)=I_0$, $G(1)=I_1$, $G_\sss=\bar s$ and $G_\ttt=\bar t$.
On the other hand,
 the small category $\C$ induces
 a functor $F$ from $\QQQ$ to $\Sets$ defined by
$F(0)=\C_0$, $F(1)=\C_1$, $F_\sss=s$ and $F_\ttt=t$.
The equations (\ref{nu:eq:diaram:s})  and (\ref{nu:eq:diaram:t})
mean 
that $\eel_{\bullet}$ is a natural transformation from $F$ to $G$
whose values at the objects $0$ and $1$ of $\QQQ$ 
are the maps $\eel_0$ and $\eel_1$, respectively.

\end{document}